\documentclass[a4paper,12pt]{article}
\usepackage{amsmath,amsthm,amsfonts,amssymb,color}
\usepackage[matrix,arrow]{xy}
\usepackage{graphicx}
\usepackage[UKenglish]{babel}
\usepackage{multicol} 
\usepackage{epstopdf}
\usepackage{hyperref}

\usepackage{caption}
\usepackage{subcaption}

\usepackage{multirow} % para las tablas
\usepackage{float}
\usepackage{graphicx}
\usepackage{longtable}  

\usepackage{epstopdf}

\usepackage{amsmath,amssymb,amsthm}
\usepackage{enumerate}
\usepackage{amsfonts}
\usepackage{graphics,graphicx}
\usepackage{verbatim}
\usepackage{latexsym,makeidx}

\usepackage{amsmath,amscd}
\usepackage{appendix}
\usepackage[latin1]{inputenc}

\newcommand{\Ste}{{\text{Ste}}}
\newcommand{\Bi}{{\text{Bi}}}
\newenvironment{keywords}{
  \vspace{2mm}
  \noindent
  \keywordsname: 
  \itshape\small
}

\newenvironment{mathsubclass}{
  \small
  \noindent
  \mathsubclassname: 
}

 \def\keywordsname{\textbf{Keywords}}
  \def\mathsubclassname{\textbf{2010 AMS Subject Classification}}

\newtheorem{teo}{Theorem}[section]

\makeatletter
\renewcommand\@biblabel[1]{#1.}
\makeatother
 
\newcommand\be{\begin{equation}}
\newcommand\ee{\end{equation}}

\begin{document}

\title{Heat balance integral methods applied to the one-phase Stefan problem with a convective boundary condition at the fixed face}

\author{
Julieta Bollati$^{1,2}$,  Jos\'e A. Semitiel$^{2}$, Domingo A. Tarzia$^{1,2}$\\ \\
\small{$^1$ Consejo Nacional de Investigaciones Cient\'ificas y Tecnol\'ogicas (CONICET)}\\
\small{$^2$ Depto. Matem\'atica - FCE, Univ. Austral, Paraguay 1950} \\  
\small{S2000FZF Rosario, Argentina.} \\
\small{Email: JBollati@austral.edu.ar; JSemitiel@austral.edu.ar; DTarzia@austral.edu.ar.}
}

\date{}

\maketitle

\begin{abstract}
In this paper we consider   a one-dimensional one-phase Stefan problem\ \  corresponding to  the solidification process of a semi-infinite material with a convective boundary condition at the fixed face. The exact solution of this problem, available recently in the literature, enable us to test the accuracy of the approximate solutions obtained by applying the classical technique of the heat balance integral method and the refined integral method, assuming a quadratic temperature profile in space. We develop variations of these methods which turn out to be optimal in some cases.  Throughout this paper,  a dimensionless analysis is carried out by using the parameters: Stefan number (Ste) and the generalized Biot number (Bi). In addition it is studied the case when Bi goes to infinity, recovering the approximate  solutions  when a Dirichlet condition is imposed at the fixed face. Some numerical simulations are provided in order to estimate the errors committed  by each approach for the corresponding free boundary and temperature profiles.
\end{abstract}

\vspace{0.03cm}

\begin{mathsubclass}
35C05, 35C06, 35K05, 35R35, 80A22.
\end{mathsubclass}

%\vspace{0.03cm}

%\begin{mathsubclass}
%35C05, 35C06, 35K05, 35R35, 80A22.
%\end{mathsubclass}

\section{Introduction}
Stefan problems model  heat transfer processes that involve a change of phase. They constitute a broad field of study since they arise in  a great number of mathematical and industrial significance problems %Phase-change problems appear frequently in industrial processes and other problems of technological interest
 \cite{AlSo1993}, \cite{Ca1984}, \cite{Gu2003}, \cite{Lu1991}. A large bibliography on the subject was given in \cite{Ta2000} and a review on analytical  solutions  is given in \cite{Ta2011}.

 Due to the non-linearity nature of this type of problems, exact solutions are limited to a few cases and it is necessary to solve them either numerically or approximately. The heat balance integral method introduced by \mbox{Goodman} in \cite{Go1958} is a well-known approximate mathematical technique for solving heat transfer problems and particularly the location of the free boundary in heat-conduction problems involving a phase of change. This method consists in transforming the heat equation into an ordinary differential equation over time by assuming a quadratic  temperature profile  in space. In \cite{Ta1990-b}, \cite{Hr2009-a}, \cite{Hr2009-b}, \cite{Mi2012}, \cite{MiMy2010-a}, \cite{MiMy2012} and in \cite{MoWoAl2002} this method is applied using  different accurate temperature profiles such as: exponential, potential, etc. Different alternative pahtways to develop the heat balance integral method were established in \cite{Wo2001}.

In the last few years, a series of papers devoted to integral method applied to a variety of thermal and moving boundary problems have been published \cite{MaMy2016}, \cite{MiOb2014}
, \cite{Hr2017-a}, \cite{Mi2015}. The recent principle problem emerging in application in heat balance integral method has spread over the area of the non-linear heat conduction: \cite{Hr2015-a}, \cite{FaHr2017}, \cite{Hr2016}, and fractional diffusion: \cite{Hr2015-b}, \cite{Hr2017-b}, \cite{Hr2017-c}.

In this paper, we obtain approximate solutions through integral heat balance methods and variants obtained thereof proposed in \cite{Wo2001} for the solidification of a semi-infinite material $x>0$  when a convective boundary condition is imposed at the fixed face $x=0$.

The convective condition states that heat flux at the fixed face is proportional to the difference between the material temperature and the neighbourhood temperature, i.e.:
\begin{equation*}
k\dfrac{\partial T}{\partial x}(0,t)=H(t)\left(T(0,t)+\Theta_{\infty} \right), 
\end{equation*}
where $T$ is the material temperature, $k$ is the thermal conductivity, $H(t)$  characterizes the heat transfer at the fixed face and $-\Theta_\infty<0$ represents the neighbourhood temperature at $x=0$.

In this paper we will consider the solidification process of a semi-infinite material when a convective condition at the fixed face $x=0$ of the form $H(t)=\frac{h}{\sqrt{t}}, h>0$ is imposed. There are very few research studies that examine the heat balance integral method applied to  Stefan problems with a convective boundary  condition.  Only in  \cite{Go1958}, \cite{MiMy2010-b} and \cite{RoKa2009} a convective condition fixing $H(t)=h>0$ is considered. Although approximate solutions are provided, their precision is verified by making comparisons with numerical methods since there is not exact solution for this choice of $H$.

The mathematical formulation of the problem under study with its corresponding exact solution given in \cite{Ta2017} will be presented in Section 2. Section 3  introduces approximate solutions using the heat balance integral method, the refined heat balance integral method and two alternatives methods for them. 
%For each $t$ fixed, if we analyse the limit of the convective condition when $h\rightarrow \infty$  we obtain a  Dirichlet condition in $x=0$ of the form: $T(0,t)=-T_\infty$. For this reason, 
In Section 4 we also study the limiting cases of the obtained approximate solutions when $h\rightarrow \infty$, recovering the approximate solutions when a temperature condition at the fixed face is imposed. Finally, in Section 5 we compare the approximate solutions  with the exact one to the problem presented in Section 1, analysing the committed error in each case. 

\section{Mathematical formulation and exact solution}

We consider a one-dimensional one-phase Stefan problem for the solidification of a semi-infinite material $x>0$, where a convective condition at the fixed face $x=0$ is imposed. This problem can be formulated mathematically in the following way:

%\textcolor{red}{SI SE PONE QUE T DEPENDA DEL NUMERO DE BIOT HAY QUE DEFINIRLO PREVIAMENTE}

\textbf{Problem (P)}. Find the temperature $T=T(x,t)$ at the solid region $0<x<s(t)$  and the location of the free boundary  $x=s(t)$ such that:
\begin{align}
 &\dfrac{\partial{T}}{\partial{t}} = \dfrac{k}{\rho c} \dfrac{\partial ^2T}{\partial x^2}~,\ 0<x<s(t)~, \ t>0~,\qquad \qquad  \qquad \qquad \qquad \qquad  \label{EcCalor}\\
& k\frac{\partial T}{\partial x}(0,t)=\dfrac{h}{\sqrt{t}} (T(0,t)+\Theta_{\infty})~,   t>0~, \label{CondConvect} \\
& T(s(t),t)= 0~,    t>0~, \label{TempFrontera}\\
&k\frac{\partial T}{\partial x}(s(t),t)=\rho \lambda \dot{s}(t)~,  t>0~, \label{CondStefan} \\
&s(0)=0~. \label{FrontInicial}
\end{align}
where the thermal conductivity  $k$, the mass density  $\rho$, the specific heat $c$  and  the latent heat per unit mass  $\lambda$  are given positive constants. The condition (\ref{CondConvect}) represents the convective condition at the fixed face where $-\Theta_{\infty}<0$ is the neighbourhood temperature at $x=0$ and $h>0$ is the coefficient that characterizes the heat transfer at the fixed face.

The analytical solution for the problem (P), using the similarity technique, was obtained in \cite{Ta2017} and for the one-phase case is given by:
\begin{eqnarray}
T(x,t)&=& -A\Theta_{\infty}+B\Theta_{\infty} \text{erf}\left( \dfrac{x}{2\sqrt{\alpha t}}\right)~, \label{TempExacta}\\
s(t)&=& 2 \xi \sqrt{\alpha t}~,\qquad (\alpha=\frac{k}{\rho c}: \text{ diffusion coefficient})\label{FrontExacta}
\end{eqnarray}
where the constants $A$ and $B$ are defined by:
\begin{eqnarray}
A&=&\dfrac{\text{erf}\left(\xi\right)}{\frac{1}{\text{Bi}\sqrt{\pi}}+ \text{erf}(\xi)}~,\label{A}\\
B&=&\dfrac{1}{\frac{1}{\text{Bi}\sqrt{\pi}}+ \text{erf}(\xi)}\label{B}~,
\end{eqnarray}
and the dimensionless coefficient $\xi$ is the unique positive solution of the following equation:
\begin{eqnarray}
z\exp{\left(z^{2}\right)}\left(\text{erf}\left(z\right)+\dfrac{1}{\text{Bi}\sqrt{\pi}}\right)-\dfrac{\text{Ste}}{\sqrt{\pi}}=0~, \qquad z>0~. \label{XiExacto}
\end{eqnarray}

The dimensionless parameters defined by:
\begin{equation} \label{alfaSteBi}
 \text{Ste}=\dfrac{c\Theta_{\infty}}{\lambda}\quad \text{and} \qquad \text{Bi}=\dfrac{h\sqrt{\alpha}}{k}~,
\end{equation}
represent the Stefan number and the generalized Biot number respectively.

\section{Approximate solutions}

As one of the mechanisms for the heat conduction is the diffusion, the excitation at the fixed face $x=0$ (for example, a temperature, a flux or a convective condition) does not spread instantaneously to the material \mbox{$x>0$}. However, the effect of the fixed boundary condition can be perceived  in a bounded interval $\left[0,\delta(t)\right]$ (for every time $t>0$) outside of which the temperature remains equal to the initial temperature. The heat balance integral method presented in \cite{Go1958} 
established the existence of a function $\delta=\delta(t)$ that measures the depth of the thermal layer. In problems with a phase of change, this layer is assumed as the free boundary, i.e $\delta(t)=s(t)$.

From equation (\ref{EcCalor}) and  conditions (\ref{TempFrontera}) and  (\ref{CondStefan}) we obtain the new condition:
\begin{equation} \label{CondStefanAprox}
\left(\dfrac{\partial T}{\partial x}\right)^2(s(t),t)=-\dfrac{\lambda}{c}\dfrac{\partial ^2 T}{\partial x^2}(s(t),t)~.
\end{equation}

From equation (\ref{EcCalor}) and condition (\ref{TempFrontera}) we obtain the integral condition:
\begin{eqnarray}
\dfrac{d}{dt} \int\limits_{0}^{s(t)} T(x,t) dx&=&\dfrac{k}{\rho c}\left[\dfrac{\rho \lambda}{k} \dot{s}(t)-\frac{\partial T}{\partial x}(0,t) \right]~.\label{EcCalorAprox}
\end{eqnarray}

The classical heat balance integral method introduced in \cite{Go1958}  to solve problem (P) proposes the resolution of a problem that arises by replacing the equation (\ref{EcCalor}) by the condition (\ref{EcCalorAprox}), and the condition (\ref{CondStefan}) by the condition (\ref{CondStefanAprox}); that is, the resolution of the approximate problem defined as follows: conditions (\ref{CondConvect}), (\ref{TempFrontera}), (\ref{FrontInicial}), (\ref{CondStefanAprox}) and (\ref{EcCalorAprox}).

In \cite{Wo2001}, a variant of the classical heat balance integral method was proposed by replacing equation (\ref{EcCalor}) by condition (\ref{EcCalorAprox}), keeping all others conditions of the problem (P) equals;  that is, the resolution of an approximate problem defined as follows: conditions (\ref{CondConvect}), (\ref{TempFrontera}), (\ref{CondStefan}), (\ref{FrontInicial}) and (\ref{EcCalorAprox}).

From equation (\ref{EcCalor}) and condition (\ref{TempFrontera}) we can also obtain:
\begin{eqnarray}
\int\limits_0^{s(t)} \int\limits_0^x \frac{\partial T}{\partial t}(\xi,t) d\xi dx &=& -\dfrac{k}{\rho c} \left[T(0,t)+\frac{\partial T}{\partial x}(0,t)s(t) \right]~.\label{EcCalorAproxRIM} 
\end{eqnarray}

The refined  heat balance integral method introduced in \cite{SaSiCo2006} to solve the problem (P) proposes the resolution of the approximate problem that arises by replacing equation (\ref{EcCalor}) by condition (\ref{EcCalorAproxRIM}), keeping all others conditions of the problem (P) equal. It is defined as follows: conditions (\ref{CondConvect}), (\ref{TempFrontera}), (\ref{CondStefan}) (\ref{FrontInicial}) and (\ref{EcCalorAproxRIM}).

From the ideas put forward in \cite{Wo2001} and inspired by \cite{SaSiCo2006}, we develop an alternative of the refined  heat balance integral method. That is, the resolution of an approximate problem defined as follows: conditions (\ref{CondConvect}), (\ref{TempFrontera}), (\ref{FrontInicial}), (\ref{CondStefanAprox}) and (\ref{EcCalorAproxRIM}).

For solving the problems previously defined  we propose a quadratic temperature profile in space as follows:
\begin{equation}\label{Perfil}
\widetilde{ T}(x,t)=-\widetilde{A}\Theta_{\infty}\left(  1-\dfrac{x}{\widetilde{s}(t)}\right)-\widetilde{B}\Theta_{\infty}\left(  1-\dfrac{x}{\widetilde{s}(t)}\right)^2,\quad 0<x<\widetilde{s}(t),~ t>0. 
\end{equation}

Taking advantage of the fact of having the exact temperature  of the \mbox{problem} (P), it is physically reasonable to impose that the approximate temperature given by (\ref{Perfil}) behaves in a similar manner than the exact one given by (\ref{TempExacta}); that is: its sign, monotony and convexity in space. As $T$ verifies $T<0$, $\frac{\partial T}{\partial x}>0$ and $\frac{\partial^2 T}{\partial x^2}<0$  on $0<x<s(t)$, $t>0$, we enforce the following conditions on $\widetilde{T}$:
\begin{eqnarray*}
&& \widetilde{T}(x,t)< 0,\\
&&\dfrac{\partial \widetilde{T}}{\partial x}(x,t) = \dfrac{\Theta_{\infty}}{\widetilde{s}(t)} \left(\widetilde{A}+2\widetilde{B}\left(1 -\dfrac{x}{\widetilde{s}(t)}\right) \right)>0, \\
&&\dfrac{\partial ^2\widetilde{T}}{\partial x^2}(x,t) =- \dfrac{2\widetilde{B}\Theta_{\infty}}{\widetilde{s}^2(t)}<0,
\end{eqnarray*}
for all $0<x<\widetilde{s}(t)$, $t>0$. Therefore, we obtain that the constants $\widetilde{A}$ and $\widetilde{B}$ must be positive.

\subsection{Approximate solution using the classical heat balance integral method}

The classical heat balance integral method in order to solve the problem (P) proposes the resolution of the approximate problem defined as follows:

\textbf{Problem (P1)}. Find the temperature $T_{1}=T_{1}(x,t)$ at the solid region $0<x<s_{1}(t)$ and the location of the free boundary $x=s_{1}(t)$ such that:
\begin{align} 
& \dfrac{d}{dt} \int\limits_{0}^{s_{1}(t)} T_{1}(x,t) dx=\dfrac{k}{\rho c}\left[\dfrac{\rho \lambda}{k} \dot{s}_{1}(t)-\frac{\partial T_1}{\partial x}(0,t) \right]~, \ 0<x<s_{1}(t)~,   t>0~, \label{EcCalorP1}\\
& k\frac{\partial T_1}{\partial x}(0,t)=\dfrac{h}{\sqrt{t}} (T_{1}(0,t)+\Theta_{\infty})~,  \ t>0~, \label{CondConvP1}\\
&T_{1}(s_{1}(t),t)= 0~,\ t>0~, \label{TempFronteraP1}\\
&\left(\frac{\partial T_1}{\partial x}\right)^2(s_{1}(t),t)=-\frac{\lambda}{c} \frac{\partial^2 T_1}{\partial x^2}(s_{1}(t),t)~,\ t>0~, \label{PseudoStefan}\\
& s_{1}(0)=0~. \label{FrontinicialP1}
\end{align} 

A solution to  problem (P1) will be an approximate one of the problem (P). Proposing the following quadratic temperature profile in space:
\begin{equation}
T_{1}(x,t)= -A_{1}\Theta_{\infty}\left(1-\dfrac{x}{s_{1}(t)}\right) -B_{1}\Theta_{\infty} \left(1-\dfrac{x}{s_{1}(t)}\right)^{2}~, \label{TempP1}
\end{equation}
the free boundary  takes the form:
\begin{equation}
s_{1}(t)=2 \xi_{1} \sqrt{\alpha t}~,\label{FrontP1}
\end{equation}
where the constants $A_{1}, B_{1}$ and $\xi_{1}$ will be determined from the conditions (\ref{EcCalorP1}), (\ref{CondConvP1}) and (\ref{PseudoStefan}). Because of  (\ref{TempP1}) and (\ref{FrontP1}), the conditions (\ref{TempFronteraP1}) and (\ref{FrontinicialP1}) are immediately satisfied. From conditions (\ref{EcCalorP1}) and (\ref{CondConvP1}) we obtain:
\begin{equation}\label{A1}
A_{1}=\dfrac{6\text{Ste}-\left(6+2\text{Ste}\right)\xi_{1}^{2}-\frac{6}{\text{Bi}}\xi_{1}}{\text{Ste}\left(\xi_{1}^{2}+\frac{2}{\text{Bi}}\xi_{1}+3\right)}~,
\end{equation}
\begin{equation}\label{B1}
B_{1}=\dfrac{\left(3\text{Ste}+6\right)\xi_{1}^{2}+\frac{3}{\text{Bi}}\xi_{1}-3\text{Ste}}{\text{Ste}\left(\xi_{1}^{2}+\frac{2}{\text{Bi}}\xi_{1}+3\right)}~.
\end{equation}

From the fact that $A_1>0$ and $B_1>0$ we obtain that $0<\xi_1<\xi^{\max}$ and $\xi_1>\xi^{\min}>0$, respectively where:
\begin{equation}\label{Xi-MinMax}
\xi^{\min}=\dfrac{ \sqrt{\Delta^{\text{min}}}-\frac{1}{\Bi} }{2(2+\Ste)}~, \qquad \qquad \xi^{\text{max}}=\dfrac{ \sqrt{\Delta^{\text{max}}}-\frac{3}{\Bi} }{2(3+\Ste)}~,
\end{equation}
with
\begin{equation}\label{Delta-MinMax}
\Delta^{\min}=4\Ste^2+8\Ste+\frac{1}{\Bi^2}~, \qquad \Delta^{\max}=12\Ste^2+36\Ste+\frac{9}{\Bi^2}.
\end{equation}

We have that $\xi^{\min}<\xi^{\max}$. In fact: 

\begin{eqnarray*}
\xi^{\max}-\xi^{\min}>0 & \Leftrightarrow & (2+\Ste)\Bi\sqrt{\Delta^{\max}}-(3+\Ste)\Bi\sqrt{\Delta^{\min}}>3+2\Ste\\
& \Leftrightarrow & \left[(2+\Ste)^2 \Bi^2\Delta^{\max}+(3+\Ste)^2 \Bi^2 \Delta^{\min}-(3+2\Ste)^2\right]^2\\
& &-4(2+\Ste)^2(3+\Ste)^2 \Bi^4 \Delta^{\min} \Delta^{\max}>0\\
&\Leftrightarrow &16\Bi^4 \Ste^2 \left(2\Ste^3+12\Ste^2+27\Ste+18 \right)^2>0,
\end{eqnarray*}
which is automatically verified.

Since $A_{1}$ and $B_{1}$ are defined from the parameters $\xi_{1}$, $\text{Bi}$ and $\text{Ste}$, condition (\ref{PseudoStefan}) will be used to find the value of $\xi_{1}$. In this way, it turns out that $\xi_{1}$ must be a positive solution of the fourth degree polynomial equation:
\begin{eqnarray} 
&\left(12+9\text{Ste}+2\text{Ste}^{2}\right)z^{4}+\frac{21+6\text{Ste}}{\text{Bi}}z^{3}+\left(\frac{12}{\text{Bi}^{2}}-42\text{Ste}-12\text{Ste}^{2}-18\right)z^{2}+ \nonumber\\
&-\frac{30\text{Ste}+9}{\text{Bi}}z+9\text{Ste}\left(1+2\text{Ste} \right)=0~, \ \  \xi_1^{\min}<z<\xi_1^{\max}~. \label{Xi-1}
\end{eqnarray}

Let us  refer to $p_1=p_1(z)$ as the polinomial function defined by the l.h.s of equation (\ref{Xi-1}). Then we focus on studying the existence of  roots of $p_1$ in the desired interval $(\xi^{\min},\xi^{\max})$.

Descartes' rule of signs states that if the terms of a single-variable polynomial with real coefficients are ordered by descending variable exponent, then the number of positive roots of the polynomial is either equal to the number of sign differences between consecutive nonzero coefficients, or is less than it by an even number. 
Therefore, in our case we can assure that $p_1$ can have at most two roots in $\mathbb{R}^+$.

In order to prove that at least one of this two positive roots belongs to the required range, $(\xi_1^{\min},\xi_1^{\max})$, we  study the sign of $p_1$ in the extremes of the interval.

On one hand we have that:
\begin{eqnarray*}
p_1(\xi^{\min})=-Q_1 \sqrt{\Delta^{\min}}+ Q_2~,
\end{eqnarray*}
where
\begin{align*}
&Q_1= \frac{(2\Ste+3)^2\left(2\Bi^2(\Ste^2+2\Ste)+1\right)}{\Bi^3(2+\Ste)^4}~, \\
&Q_2= \frac{(2\Ste+3)^2\left(\left(2\, {\mathrm{Ste}}^4 + 8\, {\mathrm{Ste}}^3 + 8\, {\mathrm{Ste}}^2\right)\, {\mathrm{Bi}}^4 + \left(4\, {\mathrm{Ste}}^2 + 8\, \mathrm{Ste}\right)\, {\mathrm{Bi}}^2 + 1\right)}{\Bi^4(2+\Ste)^4}~.
\end{align*}

It is clear that $Q_1>0$ and $Q_2>0$. Therefore 
\begin{eqnarray*}
p_1(\xi^{\min})>0\quad & \Leftrightarrow &\quad Q_2^2-Q_1^2 \Delta^{\min}>0\\
&\Leftrightarrow & \quad 1024\Ste^4\left( 2\Ste^2+7\Ste+6\right)^4>0,
\end{eqnarray*}
which is automatically verified.

On the other hand we have that:
$$p_1(\xi^{\max})=Q_3 \sqrt{\Delta^{\max}}+ Q_4~,$$
where
\begin{eqnarray*}
Q_3&=& -\frac{3(2\Ste+3)^2\left(\Bi^2\left(\Ste^2-9 \right) +3\right)}{2\Bi^3 (3+\Ste)^4}~, \\
Q_4&=& -\tfrac{9}{2}(2\Ste+3)^2 \left[ 2\Bi^4 \Ste^3+ 3\Bi^2 (4\Bi^2-1)\Ste^2+ \right. \\
&+&\left.6\Bi^2 (3\Bi^2-1)\Ste+3(3\Bi^2-1)\right] \Bi^{-4} (3+\Ste)^{-4}~.
\end{eqnarray*}

An easy computation shows that:
\begin{equation*}
Q_4^2-\Delta^{\text{max}} Q_3^2=6912 \Ste^2 (3\Bi^2-1)(2\Ste^2+9\Ste+9)^4.
\end{equation*}

Therefore, it is clear that the following properties are satisfied:

\begin{enumerate}[a.]
\item If $\Bi<\frac{\sqrt{3}}{3}$ then:
\begin{enumerate}[i) ]
\item  $Q_3<0, \quad \forall\ \Ste>0$.
\item  $Q_4^2-\Delta^{\max}Q_3^2<0,\quad \forall \ \Ste>0$.
\end{enumerate}
\item If $\Bi>\frac{\sqrt{3}}{3}$ then:
\begin{enumerate}[i) ]
\item  $Q_3>0\ \text{ if } \ \Ste<\sqrt{9-\frac{3}{\Bi^2}}$ \quad and \quad $Q_3<0\ \text{ if } \ \Ste>\sqrt{9-\frac{3}{\Bi^2}}$.
\item $Q_4<0, \quad \forall\ \Ste>0$.
\item  $Q_4^2-\Delta^{\max}Q_3^2>0,\quad \forall\  \Ste>0$.
\end{enumerate}
\item If $\Bi=\frac{\sqrt{3}}{3}$ then: 
\begin{enumerate}[i) ]
\item $Q_3<0, \quad \forall\ \Ste>0$.
\item $Q_4<0, \quad \forall\ \Ste>0$.
\end{enumerate}

\end{enumerate}
Then, let us prove that $p_1(\xi^{\max})<0$, $\forall \ \Ste>0$, $\forall \ \Bi>0$.

In case $\Bi<\frac{\sqrt{3}}{3}$, we have from property a.~i) that $Q_3<0$.

\quad  If $Q_4\leq 0$, then $p_1(\xi^{\max})<0$ immediately.

\quad  If $Q_4>0$ then
\begin{eqnarray*}
p_1(\xi^{\max})<0 &\Leftrightarrow &Q_3\sqrt{\Delta^{\max}}+Q_4<0\\
                  &\Leftrightarrow &Q_3\sqrt{\Delta^{\max}}<-Q_4<0\\
                  &\Leftrightarrow & Q_4^2-\Delta^{\max}Q_3^2<0,
\end{eqnarray*}
\quad \quad which is immediately verified from property a.~ii).

In case $\Bi>\frac{\sqrt{3}}{3}$, we have from property b.~ii) that $Q_4<0$. 

\quad  If $Q_3\leq 0$, then $p_1(\xi^{\max})<0$ immediately.

\quad  If $Q_3>0$ then
\begin{eqnarray*}
p_1(\xi^{\max})<0 &\Leftrightarrow &Q_3\sqrt{\Delta^{\max}}+Q_4<0\\
                  &\Leftrightarrow &0<Q_3\sqrt{\Delta^{\max}}<-Q_4\\
                  &\Leftrightarrow & Q_4^2-\Delta^{\max}Q_3^2>0,
\end{eqnarray*}
\quad \quad which is immediately verified from property b.~iii).
\vspace{0.15cm}

In case $\Bi=\frac{\sqrt{3}}{3}$, we have from properties c.~i) and c.~ii) that $Q_3<0$ and $Q_4<0$ which obviously imply that $p_1(\xi^{\max})<0$.

So far we can claim that  $p_1(\xi^{\min})>0$ and $p_1(\xi^{\max})<0$, $\forall~ \Ste>0$,\ \mbox{$\forall~ \Bi>0$}. In addition, from the fact that $p_1$ has at most two roots in $\mathbb{R}^+$ and \mbox{$p_1(+\infty)=+\infty$}, we can conclude that $p_1$ has exactly one root on the interval $(\xi^{\min},\xi^{\max})$.

All the above analysis can be summarized in the following theorem:

\begin{teo} The solution to the problem \emph{(P1)}, for a quadratic profile in space, is given by (\ref{TempP1})-(\ref{FrontP1}), where the positive constants $A_1$ and $B_1$ are defined by (\ref{A1}) and (\ref{B1}) respectively and $\xi_1$ is the unique solution of the polynomial equation (\ref{Xi-1}) where $\xi^{\min}$ and $\xi^{\max}$ are defined in (\ref{Xi-MinMax}).

\end{teo}

\vspace{0.10cm}
\subsection{Approximate solution using an alternative of the heat balance integral method}

An alternative method of the classical heat balance integral method in order to solve the problem (P) proposes the resolution of the approximate problem defined as follows:

\textbf{Problem (P2)}. Find the temperature $T_{2}=T_{2}(x,t)$ at the solid region $0<x<s_{2}(t)$ and  the location of the free boundary $x=s_{2}(t)$ such that:
\begin{align}
&\dfrac{d}{dt} \int\limits_{0}^{s_{2}(t)} T_{2}(x,t) dx=\dfrac{k}{\rho c}\left[\dfrac{\rho \lambda}{k} \dot{s}_{2}(t)-\frac{\partial T_2}{\partial x}(0,t) \right]~, \ 0<x<s_{2}(t)~, \ t>0~, \label{EcCalorP2}\\
& k\frac{\partial T_2}{\partial x}(0,t)=\dfrac{h}{\sqrt{t}} (T_{2}(0,t)+\Theta_{\infty})~, \ t>0~, \label{CondConvP2}\\
& T_{2}(s_{2}(t),t)= 0~, \ t>0~, \label{TempFronteraP2}\\
& k\frac{\partial T_2}{\partial x}(s_{2}(t),t)=\rho \lambda \dot{s}_{2}(t)~, \ t>0~, \label{CondStefanP2} \\
& s_{2}(0)=0~. \label{FrontinicialP2}
\end{align} 

A solution to the problem (P2), for a quadratic temperature profile in space, is obtained by
\begin{equation}
T_{2}(x,t)= -A_{2}\Theta_{\infty}\left(1-\dfrac{x}{s_{2}(t)}\right) -B_{2}\Theta_{\infty} \left(1-\dfrac{x}{s_{2}(t)}\right)^{2}~, \label{TempP2}
\end{equation}
\begin{equation}
s_{2}(t)=2 \xi_{2} \sqrt{\alpha t}~,\label{FrontP2}
\end{equation}
where the constants $A_{2}, B_{2}$ y $\xi_{2}$ will be determined from the conditions (\ref{EcCalorP2}), (\ref{CondConvP2}) and (\ref{CondStefanP2}) of the problem (P2). The conditions (\ref{TempFronteraP2}) and (\ref{FrontinicialP2}) are immediately satisfied. From conditions (\ref{EcCalorP2}) and (\ref{CondConvP2}), we obtain:
\begin{equation}\label{A2}
A_{2}=\dfrac{6\text{Ste}-\left(6+2\text{Ste}\right)\xi_{2}^{2}-\frac{6}{\text{Bi}}\xi_{2}}{\text{Ste}\left(\xi_{2}^{2}+\frac{2}{\text{Bi}}\xi_{2}+3\right)}~,
\end{equation}
\begin{equation}\label{B2}
B_{2}=\dfrac{\left(3\text{Ste}+6\right)\xi_{2}^{2}+\frac{3}{\text{Bi}}\xi_{2}-3\text{Ste}}{\text{Ste}\left(\xi_{2}^{2}+\frac{2}{\text{Bi}}\xi_{2}+3\right)}~.
\end{equation}

Since $A_2$ and $B_2$ must be positive  we obtain, as in problem (P1), that $0<\xi^{\min}<\xi_2<\xi^{\max}$ where $\xi^{\min}$  and $\xi^{\max}$  are defined by (\ref{Xi-MinMax}).

The constants $A_{2}$ and $B_{2}$, are expressed as a function of the parameters
$\xi_{2}$, $\text{Bi}$ and $\text{Ste}$. By using condition (\ref{CondStefanP2}), the coefficient $\xi_{2}$ is a  positive solution of the fourth degree polynomial equation given by:
\begin{eqnarray}
z^{4}+\frac{2}{\text{Bi}}z^{3}+\left(6+\text{Ste}\right)z^{2}+\frac{3}{\text{Bi}}z-3\text{Ste}=0~, \qquad \quad \xi^{\min}<z<\xi^{\max}. \label{Xi-2}
\end{eqnarray}

Notice that the polynomial function $p_2=p_2(z)$ defined by the l.h.s of equation (\ref{Xi-2}) is an increasing function in $\mathbb{R}^+$ that assumes a negative value at $z=0$ and goes to $+\infty$ when $z$ goes to $+\infty$. Hence, $p_2$ has a unique root in $\mathbb{R}^{+}$.
  
If we analyse the behaviour  of $p_2$ on the interval $(\xi^{\min}, \xi^{\max})$ we can observe that:
$$p_2(\xi^{\min})=R_1 \sqrt{\Delta^{\min}}-R_2$$
where $\Delta^{\min}$ is defined by equation (\ref{Delta-MinMax}) and 
\begin{eqnarray*}
R_1&=& \dfrac{(2\Ste+3)(2\Bi^2\Ste^2+4\Bi^2\Ste+1)}{2\Bi^3 (2+\Ste)^4}>0~,\\
R_2&=&\left(\dfrac{\Ste^2(2\Ste+3)}{(2+\Ste)^2}+\dfrac{2\Ste (2\Ste^2+7\Ste+6)}{\Bi^2(2+\Ste)^4}+\dfrac{2\Ste+3}{2\Bi^4 (2+\Ste)^4} \right)>0~, 
\end{eqnarray*}
then
\begin{eqnarray*}
p_2(\xi^{\min})<0\quad & \Leftrightarrow &\quad R_2^2-R_1^2 \Delta^{\min}>0\\
&\Leftrightarrow & \quad 256 \Ste^4(2\Ste+3)^2(\Ste+2)^4>0,
\end{eqnarray*}
which is immediately satisfied.

Notice that $p_2(z)>\widehat{p}_2(z)$ for all $z \in \mathbb{R}^+$, where $\widehat{p}_2(z)=(6+\Ste)z^2+\frac{3}{\Bi}z-3\Ste$. Furthermore 
$$\widehat{p}_2(\xi^{\max})=-R_3\sqrt{\Delta^{\max}}+R_4$$
where $\Delta^{\max}$ is defined by equation (\ref{Delta-MinMax}) and 
\begin{eqnarray*}
R_3&=& \dfrac{9}{2\Bi(3+\Ste)^2}>0,\\
R_4&=& \dfrac{9\Ste}{(3+\Ste)}+\dfrac{27}{2 \Bi^2(3+\Ste)^2}>0.
\end{eqnarray*}

Then,
$$p_2\left(\xi^{\max}\right)>\widehat{p_2}\left(\xi^{\max}\right)>0  \Leftrightarrow  R_4^2-R_3^2\Delta^{\max}>0  \Leftrightarrow  1296 \Ste^2(\Ste+3)^2>0,$$
which is automatically verified.

As consequence, equation (\ref{Xi-2}) has a unique solution $\xi_2$ in the interval $(\xi^{\min}, \xi^{\max})$. Therefore, we have proved the following theorem:

\begin{teo} The solution to the problem \emph{(P2)}, for a quadratic profile in space, is given by (\ref{TempP2})-(\ref{FrontP2}), where the positive constants $A_2$ and $B_2$ are defined by (\ref{A2}) and (\ref{B2}) respectively and $\xi_2$ is the unique solution of the polynomial equation (\ref{Xi-2}) where $\xi^{\min}$ and $\xi^{\max}$ are defined in (\ref{Xi-MinMax}).

\end{teo}

\vspace{0.1cm}
\subsection{Approximate solution using the refined heat balance integral method}
The refined heat balance integral method in order to solve the problem (P), proposes the resolution of an approximate problem formulated as follows:

\textbf{Problem (P3)}. Find the temperature $T_{3}=T_{3}(x,t)$ at the solid region $0<x<s_{3}(t)$ and the location of the free boundary $x=s_{3}(t)$ such that:
\begin{align} 
& \int\limits_0^{s_{3}(t)} \int\limits_0^x \frac{\partial T_3}{\partial t} (\xi,t) d\xi dx =  -\dfrac{k}{\rho c} \left[T_{3}(0,t)+ \right. \qquad \qquad  \qquad  \qquad \qquad \qquad \qquad \nonumber \\
&\qquad \qquad  \qquad \quad \qquad +\left.\frac{\partial T_3}{\partial x}(0,t)s_{3}(t) \right], \ 0<x<s_3(t), \ t>0, \label{EcCalorP3}\\
& k\dfrac{\partial T_3}{\partial x}(0,t)=\dfrac{h}{\sqrt{t}} (T_{3}(0,t)+\Theta_{\infty})~, \ t>0~, \label{CondConvP3}\\
& T_{3}(s_{3}(t),t)= 0~, \ t>0~, \label{TempFronteraP3}\\
 &k\dfrac{\partial T_3}{\partial x}(s_{3}(t),t)=\rho \lambda \dot{s}_{3}(t)~, \ t>0~, \label{CondStefanP3} \\
& s_{3}(0)=0~. \label{FrontinicialP3}
\end{align}

A solution to the problem (P3) for a quadratic temperature profile in space is given by:
\begin{equation}
T_{3}(x,t)=-A_{3}\Theta_{\infty}\left(1-\dfrac{x}{s_{3}(t)}\right) -B_{3}\Theta_{\infty} \left(1-\dfrac{x}{s_{3}(t)}\right)^{2}~, \label{TempP3}
\end{equation}
and the free boundary is obtained of the form:
\begin{equation}
s_{3}(t)=2 \xi_{3} \sqrt{\alpha t}~,\label{FrontP3}
\end{equation}
where the constants $A_{3}, B_{3}$ y $\xi_{3}$ will be determined from the conditions (\ref{EcCalorP3}), (\ref{CondConvP3}) and (\ref{CondStefanP3}) of the problem (P3). From conditions (\ref{EcCalorP3}) and (\ref{CondConvP3}) we obtain:
\begin{equation}\label{A3}
A_{3}=\dfrac{2\xi_{3}(3-\xi_3^{2})}{\frac{1}{\text{Bi}}\xi_{3}^{2}+6\xi_{3}+\frac{3}{\text{Bi}}}~,
\end{equation}
\begin{equation}\label{B3}
B_{3}=\dfrac{2\xi_{3}^{3}}{\frac{1}{\text{Bi}}\xi_{3}^{2}+6\xi_{3}+\frac{3}{\text{Bi}}}~.
\end{equation}

As is already know $A_3$ and $B_3$ must be positive thus we obtain that $0<\xi_3<\sqrt{3}$. On the other hand, since $A_{3}$ and $B_{3}$ are defined from the parameters $\xi_{3}$ and $\text{Bi}$, condition (\ref{CondStefanP3}) will be used to find the value of $\xi_{3}$. In this way it turns out that $\xi_{3}$ is a positive solution of the third degree polynomial equation:
 \begin{eqnarray}
\frac{1}{\text{Bi}}z^{3}+\left(6+\text{Ste}\right)z^{2}+\frac{3}{\text{Bi}}z-3\text{Ste}=0~, \qquad 0<z<\sqrt{3}~. \label{Xi-3}
\end{eqnarray}

It is clear that the polynomial function $p_3=p_3(z)$ defined by the l.h.s of equation (\ref{Xi-3}) has a unique root in $\mathbb{R}^{+}$. Moreover, 
since we have that  
  \begin{eqnarray*}
p_3(0)&=&-3\Ste<0,\\
p_3(\sqrt{3})&=&\frac{6\sqrt{3}}{\Bi}+18>0,
\end{eqnarray*} 
we can assure that  the unique positive solution $\xi_3$ to equation (\ref{Xi-3})  belongs to  the range $(0, \sqrt{3})$. 

Therefore, we have proved the following theorem:

\begin{teo} The solution to the problem \emph{(P3)}, for a quadratic profile in space, is given by (\ref{TempP3})-(\ref{FrontP3}), where the positive constants $A_3$ and $B_3$ are defined by (\ref{A3}) and (\ref{B3}) respectively and $\xi_3$ is the unique solution of the polynomial equation (\ref{Xi-3}).

\end{teo}

\subsection{Approximate solution using an alternative of the refined heat balance integral method}
On this subsection, we develop an alternative  of the refined heat balance integral method to solve the problem (P). This method may consist on  the resolution of the approximate problem defined as follows:

\textbf{Problem (P4)}. Find the temperature $T_{4}=T_{4}(x,t)$ at the solid region $0<x<s_{4}(t)$ and the location of the free boundary $x=s_{4}(t)$ such that:
\begin{align} 
& \int\limits_0^{s_{4}(t)} \int\limits_0^x \frac{\partial T_4}{\partial t} (\xi,t) d\xi dx =  -\dfrac{k}{\rho c} \left[T_{4}(0,t)+ \right. \qquad \qquad  \qquad  \qquad \qquad \qquad \qquad \nonumber \\
&\qquad \qquad  \qquad \quad \qquad +\left.\frac{\partial T_4}{\partial x}(0,t)s_{4}(t) \right], \ 0<x<s_4(t), \ t>0, \label{EcCalorP4}\\
& k\dfrac{\partial T_4}{\partial x}(0,t)=\dfrac{h}{\sqrt{t}} (T_{4}(0,t)+\Theta_{\infty})~, \ t>0~, \label{CondConvP4}\\
& T_{4}(s_{4}(t),t)= 0~, \ t>0~, \label{TempFronteraP4}\\
& \left(\dfrac{\partial T_4}{\partial x}\right)^2(s_{4}(t),t)=-\frac{\lambda}{c}\dfrac{\partial ^2 T_4}{\partial x^2}(s_{4}(t),t)~, \ t>0~, \label{PseudoStefan4}\\
& s_{4}(0)=0~. \label{FrontinicialP4}
\end{align}

A solution to the problem (P4) for a quadratic temperature profile in space is given by:
\begin{equation}
T_{4}(x,t)= -A_{4}\Theta_{\infty}\left(1-\dfrac{x}{s_{4}(t)}\right) -B_{4}\Theta_{\infty} \left(1-\dfrac{x}{s_{4}(t)}\right)^{2}~, \label{TempP4}
\end{equation}
and the free boundary is obtained of the form:
\begin{equation}
s_{4}(t)=2 \xi_{4} \sqrt{\alpha t}~,\label{FrontP4}
\end{equation}
where the constants $A_{4}, B_{4}$ y $\xi_{4}$ will be determined from the conditions (\ref{EcCalorP4}), (\ref{CondConvP4}) and (\ref{PseudoStefan4}). From conditions (\ref{EcCalorP4}) and (\ref{CondConvP4}) we obtain:
\begin{equation}\label{A4}
A_{4}=\dfrac{2\xi_{4}(3-\xi_4^{2})}{\frac{1}{\text{Bi}}\xi_{4}^{2}+6\xi_{4}+\frac{3}{\text{Bi}}}~,
\end{equation}
\begin{equation}\label{B4}
B_{4}=\dfrac{2\xi_{4}^{3}}{\frac{1}{\text{Bi}}\xi_{4}^{2}+6\xi_{4}+\frac{3}{\text{Bi}}}~.
\end{equation}

As we know, $A_4$ and $B_4$ must be positive, thus $0<\xi_4<\sqrt{3}$. Moreover, since $A_{4}$ and $B_{4}$ are defined from the parameters $\xi_{4}$ and $\text{Bi}$, condition (\ref{PseudoStefan4}) is used to find the value of $\xi_{4}$. In this way, it turns out that $\xi_{4}$ is a positive solution of the fourth degree polynomial equation:
 \begin{eqnarray}
\text{Ste}\ z^{4}-\frac{1}{\text{Bi}}z^{3}-6\left(1+\text{Ste} \right)z^2-\frac{3}{\text{ Bi}}z+9\text{Ste}=0~, \quad 0<z<\sqrt{3}~. \label{Xi-4}
\end{eqnarray}

Notice that using Descartes' rule of signs, we can assure that the polinomial function $p_4=p_4(z)$ defined by the l.h.s of equation (\ref{Xi-4}) can have at most two possible roots in $\mathbb{R}^+$.

If we restrict the domain of $p_4$ to the interval $(0,\sqrt{3})$ we can observe that:
\begin{eqnarray*}
p_4(0)&=&9\Ste>0,\\
p_4(\sqrt{3})&=&-\dfrac{6\sqrt{3}}{\Bi}-18<0.
\end{eqnarray*}

On the other hand we can see that $p_4(+\infty)=+\infty$. Therefore the equation (\ref{Xi-4}) has exactly two different solutions in $\mathbb{R}^+$ and a unique solution $\xi_4$ in the interval $(0, \sqrt{3})$. Then, we have the following theorem:

\begin{teo} The solution to the problem \emph{(P4)}, for a quadratic profile in space, is given by (\ref{TempP4})-(\ref{FrontP4}), where the positive constants $A_4$ and $B_4$ are defined by (\ref{A4}) and (\ref{B4}) respectively and $\xi_4$ is the unique solution of the polynomial equation (\ref{Xi-4}).

\end{teo}

\section{Analysis of the approximate solutions when Bi$\rightarrow \infty$}

\par In  problem (P), a convective boundary condition (\ref{CondConvect}) characterized by the coefficient $h$ at the fixed face $x=0$ is imposed. This condition constitutes a generalization of the Dirichlet one in the sense that if we take the limit when $h\rightarrow\infty$ we must obtain $T(0,t)=-\Theta_{\infty}~.$ From definition (\ref{alfaSteBi}), studying the limit behaviour of the solution to our problem   when $h\rightarrow\infty$ is equivalent to study the case when Bi$\rightarrow\infty$. 

  In \cite{Ta2017} it was proved that the solution to  problem (P) when $h$ and  so Bi goes to infinity converges to the solution to the following problem:

\textbf{Problem (P$_{\infty}$)}. Find the temperature $T_{\infty}=T_{\infty}(x,t)$ at the solid region $0<x<s_{\infty}(t)$  and the location of the free boundary $x=s_{\infty}(t)$ such that:
\begin{align}
&\dfrac{\partial T_{\infty}}{\partial t} = \dfrac{k}{\rho c} \dfrac{\partial^2 T_{\infty}}{\partial x^2}~,\ 0<x<s_{\infty}(t)~, \ t>0~,\qquad \qquad \qquad \qquad \qquad  \\
& T_{\infty}(0,t)=-\Theta_{\infty}~,   t>0~,  \\
& T_{\infty}(s_{\infty}(t),t)= 0~,    t>0~, \\
&k\dfrac{\partial T_{\infty}}{\partial x}(s_{\infty}(t),t)=\rho \lambda \dot{s}_{\infty}(t)~,  t>0~,  \\
& s_{\infty}(0)=0~. 
\end{align}
whose exact solution, using the similarity technique, is given in \cite{AlSo1993}, \cite{Lu1991},\cite{Ta2011} by:
\begin{eqnarray}
T_{\infty}(x,t)&=&-\Theta_{\infty} +\frac{\Theta_{\infty}}{\text{erf}(\xi_{\infty})} \text{erf}\left( \dfrac{x}{2\sqrt{\alpha t}}\right)~, \label{TempPinf}\\
s_{\infty}(t)&=& 2 \xi_{\infty} \sqrt{\alpha t}~,\label{FrontPInf}
\end{eqnarray}
where $\xi_{\infty}$ is the unique positive solution of the following equation:
\begin{eqnarray}
z\exp{\left(z^{2}\right)}\text{erf}\left(z\right)-\dfrac{\text{Ste}}{\sqrt{\pi}}=0~, \qquad z>0~. \label{XiInf}
\end{eqnarray}

Notice that the solution $T(x,t),~s(t)$ to problem (P)  given by formulas (\ref{TempExacta}), (\ref{FrontExacta}), depend on the parameters $\xi,A$ and $B$  which in turns depend on the parameters Ste and Bi. Therefore, by  ``convergence" of the solution of (P) to the solution of (P$_{\infty}$) it is understood that:  for every Ste$,~x,~t>0$ fixed: $\xi\rightarrow \xi_{\infty}~$, $A\rightarrow A_{\infty}~$, $B\rightarrow B_{\infty}$ when Bi$\rightarrow \infty$. In this way it turns out that $T(x,t)\rightarrow T_{\infty}(x,t)$ and $s(t)\rightarrow s_{\infty}(t)$  is immediately verified when Bi$\rightarrow\infty~$.

Motivated by the previous ideas, we  devote this section in the analysis of the limit behaviour of the solution of each approximate problem (Pi), $i=1,2,3,4$ when Bi goes to infinity. We will prove that the solution of each  (Pi) converge to the solution of a new problem (Pi$_{\infty}$) that is defined from (Pi) $i=1,2,3,4$ by changing the convective boundary condition on $x=0$, by a Dirichlet condition, as follows:

\textbf{Problem (P${1_{\infty}}$)}. Find the temperature $T_{1\infty}=T_{1\infty}(x,t)$ at the solid region $0<x<s_{1\infty}(t)$  and the location of the free boundary  $x=s_{1\infty}(t)$ such that:
\begin{align}
&\dfrac{d}{dt}\int\limits_{0}^{s_{1 \infty}(t)} T_{1\infty}(x,t) dx=\dfrac{k}{\rho c}\left[\dfrac{\rho \lambda}{k} \dot{s}_{1\infty}(t)- \right.\qquad\qquad\qquad\qquad\qquad\qquad\nonumber\\ 
& \qquad\qquad\qquad\quad\qquad \left.+\dfrac{\partial T_{1\infty}}{\partial x}(0,t) \right], \ 0<x<s_{1\infty}(t), \ t>0~, \label{EcCalorP1Inf}\\
&T_{1\infty}(0,t)=-\Theta_{\infty}~,   t>0~,\\
&T_{1\infty}(s_{1\infty}(t),t)= 0~,    t>0~, \\
&\left(\dfrac{\partial T_{1 \infty}}{\partial x}\right)^2(s_{1\infty}(t),t)=-\frac{\lambda}{c}\dfrac{\partial^2 T_{1\infty}}{\partial x^2}(s_{1\infty}(t),t)~, \ t>0~,\\
&s_{1\infty}(0)=0~. 
\end{align}

The solution to  problem (P1$_{\infty}$) obtained in \cite{Go1958} is given by:
\begin{eqnarray}
T_{1\infty}(x,t)&=&-A_{1\infty}\Theta_{\infty}\left(1-\dfrac{x}{s_{1\infty}(t)}\right) -B_{1\infty}\Theta_{\infty} \left(1-\dfrac{x}{s_{1\infty}(t)}\right)^{2}, \label{TempP1Inf}\\
s_{1\infty}(t)&=&2 \xi_{1\infty} \sqrt{\alpha t},\label{FrontP1Inf}
\end{eqnarray}
where the constants $A_{1\infty}$ and $B_{1\infty}$ are given by:
\begin{eqnarray}
A_{1\infty}&=&\dfrac{6\text{Ste}-\left(6+2\text{Ste}\right)\xi_{1\infty}^{2}}{\text{Ste}\left(\xi_{1\infty}^{2}+3\right)}~,\label{A1inf}\\ 
B_{1\infty}&=&\dfrac{\left(3\text{Ste}+6\right)\xi_{1\infty}^{2}-3\text{Ste}}{\text{Ste}\left(\xi_{1\infty}^{2}+3\right)}~, \label{B1inf}
\end{eqnarray}

In order that $A_{1\infty}$ and $B_{1\infty}$ be positive we need that \mbox{$0<\xi_{\infty}^{\min}<\xi_{1\infty}<\xi_{\infty}^{\max}$} where:
\begin{equation}\label{Xi-InfMinMax}
\xi_{\infty}^{\min}=\sqrt{\dfrac{\Ste}{2+\Ste}} \qquad \text{ and }\qquad \xi_{\infty}^{\max}=\sqrt{\dfrac{3\Ste}{3+\Ste}}.
\end{equation}

Then  $\xi_{1\infty}$ must be a positive solution of the fourth degree polynomial equation:
\begin{eqnarray} 
\left(12+9\text{Ste}+2\text{Ste}^{2}\right)z^{4}-\left(42\text{Ste}+12\text{Ste}^{2}+18\right)z^{2}+ \nonumber \\
+9\text{Ste}(1+2\text{Ste})=0~, \,\, \xi_{\infty}^{\min}<z<\xi_{\infty}^{\max}. \label{Xi-1Inf}
\end{eqnarray}

Notice that if we define define by \mbox{$p_{1\infty}=p_{1\infty}(z)$} the l.h.s of  equation (\ref{Xi-1Inf}), we obtain:
\begin{align*}
p_{1\infty}(\xi_{\infty}^{\min})&=\dfrac{2\Ste^2 (2\Ste+3)^2}{(\Ste+2)^2}>0, \\
p_{1\infty}(\xi_{\infty}^{\max})&=\dfrac{-9\Ste(2\Ste+3)^2}{(\Ste+3)^2}<0. 
\end{align*}

Therefore due to the fact that (\ref{Xi-1Inf}) has exactly two positive solutions and $p_{1\infty}(+\infty)=+\infty$, we can assure that there is one and only one root of $p_{1\infty}$ in the interval $(\xi_{\infty}^{\min},\xi_{\infty}^{\max})$.
That means:
\begin{equation}\label{Xi-1Inf-valor2}
\xi_{1\infty}=\left(\frac{3(2\text{Ste}+1)(\text{Ste}+3)-(9+6\text{Ste})\sqrt{2\text{Ste}+1}}{12+9\text{Ste}+2\text{Ste}^2}\right)^{1/2}~.
\end{equation}

Notice that this value of $\xi_{1\infty}$ leads us to define $A_{1\infty}$ and $B_{1\infty}$ in an explicit form:

\begin{equation}\label{A1B1Inf}
A_{1\infty}=\dfrac{\sqrt{2\Ste+1}-1}{\Ste},  \qquad B_{1\infty}=1+\dfrac{1-\sqrt{2\Ste+1}}{\Ste}
\end{equation}

The above reasoning can be summarized in the following theorem:

\begin{teo} The solution to the problem \emph{(P1$_\infty$)}, for a quadratic profile in space, is unique and it is given by (\ref{TempP1Inf}) and (\ref{FrontP1Inf}) where the positive constants $A_{1\infty}$, $B_{1\infty}$ and $\xi_{1\infty}$ are defined explicitly from the data of the problem by the expressions (\ref{A1B1Inf}) and (\ref{Xi-1Inf-valor2}) respectively.

\end{teo}

Once we obtain the solution to (P1$_{\infty}$) we can state the following convergence result:

\begin{teo}\label{Teo:ConvP1} The solution to problem \emph{(P1)} converges to the solution to  problem \emph{(P$1_{\infty}$)} when \emph{Bi}$ \rightarrow \infty$. In this case, by convergence will be understand that for a fixed $\emph{Ste}>0$, $T_1(x,t)\rightarrow T_{1\infty}(x,t)$ and $s_1(t)\rightarrow s_{1\infty}(t)$ when  $\emph{Bi}\to\infty$, for every $0<x<s_{1\infty}(t)$ and $t>0$ .
\end{teo}
\begin{proof}

From problem (P1), if we fixed  $\Ste>0$, we know that \hbox{$\xi_1=\xi_1(\text{Bi})$} is the unique solution to equation (\ref{Xi-1}) in the interval $(\xi^{\min},\xi^{\max})$. If we take the limit when Bi$\rightarrow \infty$ we obtain that $\lim\limits_{\text{Bi} \to \infty}{\xi_{1}(\text{Bi})}$ must be a solution of equation (\ref{Xi-1Inf}) in the interval $(\xi_{\infty}^{\min},\xi_{\infty}^{\max})$. As the latter equation has a unique solution given by $\xi_{1\infty}$ defined by (\ref{Xi-1Inf-valor2}) it results $\lim\limits_{\text{Bi} \to \infty}{\xi_{1}(\text{Bi})}=\xi_{1\infty}$. It follows immediately that $\lim\limits_{\text{Bi}\rightarrow\infty}s_1(t,\text{Bi})=s_{1\infty}(t)~$. For the convergence of $T_1(x,t,\text{Bi})$ we prove by simple computations  that $A_1(\text{Bi})\rightarrow A_{1\infty}$ and $B_1(\text{Bi})\rightarrow B_{1\infty},$ when $\Bi\to\infty$. 
\end{proof}

\textbf{Problem (P${2_{\infty}}$)}. Find the temperature $T_{2\infty}=T_{2\infty}(x,t)$ at the solid region $0<x<s_{2\infty}(t)$  and the location of the free boundary $x=s_{2\infty}(t)$ such that:
\begin{align}
&\dfrac{d}{dt} \int\limits_{0}^{s_{2\infty}(t)} T_{2\infty}(x,t) dx=\dfrac{k}{\rho c}\left[\dfrac{\rho \lambda}{k} \dot{s}_{2\infty}(t)-\right. \qquad \qquad\qquad\qquad\qquad\qquad\qquad\nonumber\\
& \qquad\qquad\qquad\qquad\quad+\left.\dfrac{\partial T_{2\infty}}{\partial x}(0,t) \right], \ 0<x<s_{2\infty}(t)~, \ t>0~, \\
&T_{2\infty}(0,t)=-\Theta_{\infty}~,   t>0~,  \\
&T_{2\infty}(s_{2\infty}(t),t)= 0~,    t>0~, \\
&k\dfrac{\partial T_{2\infty}}{\partial x}(s_{2\infty}(t),t)=\rho \lambda \dot{s}_{2\infty}(t)~, \ t>0~,  \\
&s_{2\infty}(0)=0~.
\end{align}

The solution to problem (P2$_{\infty}$), obtained in \cite{Wo2001}, is given by:
\begin{eqnarray}
T_{2\infty}(x,t)&=&-A_{2\infty}\Theta_{\infty}\left(1-\dfrac{x}{s_{2\infty}(t)}\right) -B_{2\infty}\Theta_{\infty} \left(1-\dfrac{x}{s_{2\infty}(t)}\right)^{2}~, \label{TempP2Inf}\\
s_{2\infty}(t)&=&2 \xi_{2\infty} \sqrt{\alpha t}~,\label{FrontP2Inf}
\end{eqnarray}
where the constants $A_{2\infty}$ and $B_{2\infty}$ are given by:
\begin{eqnarray}
A_{2\infty}&=&\dfrac{6\text{Ste}-\left(6+2\text{Ste}\right)\xi_{2\infty}^{2}}{\text{Ste}\left(\xi_{2\infty}^{2}+3\right)}~,\\
B_{2\infty}&=&\dfrac{\left(3\text{Ste}+6\right)\xi_{2\infty}^{2}-3\text{Ste}}{\text{Ste}\left(\xi_{2\infty}^{2}+3\right)}~,
\end{eqnarray}

In order that $A_{2\infty}$ and $B_{2\infty}$ be positive, as in the problem (P1), we need that $0<\xi_{\infty}^{\min}<\xi_{2\infty}<\xi_{\infty}^{\max}$ where $\xi_{\infty}^{\min}$ and $\xi_{\infty}^{\max}$ are given in formulas (\ref{Xi-InfMinMax}).
Consequently $\xi_{2\infty}$ must be a positive solution of the fourth degree polynomial equation:
\begin{eqnarray} 
z^{4}+\left(6+\text{Ste}\right)z^{2}-3\text{Ste}=0~, \qquad ~\xi_{\infty}^{\min}<z<\xi_{\infty}^{\max}, \label{Xi-2Inf}
\end{eqnarray}

Notice first that the polynomial function \mbox{$p_{2\infty}=p_{2\infty}(z)$}  defined by the l.h.s of  equation (\ref{Xi-2Inf}), has a unique root in $\mathbb{R}^{+}$.
In addition as we have
\begin{align*}
p_{2\infty}(\xi_{\infty}^{\min})&=\dfrac{-\Ste^2(2\Ste+3)}{(\Ste+2)^2}<0, \\
p_{2\infty}(\xi_{\infty}^{\max})&=\dfrac{9\Ste(2\Ste+3)^2}{(\Ste+3)^2}>0,
\end{align*}
we can assure that the unique positive root of $p_{2\infty}$ is located on the interval $(\xi_{\infty}^{\min},\xi_{\infty}^{\max})$ and it can be explicitly obtained by the expression:
\begin{equation}\label{Xi-2Inf-valor2}
\xi_{2\infty}=\left(\dfrac{\sqrt{(6+\Ste)^2+12\Ste}-(6+\Ste)}{2}\right)^{1/2}~.
\end{equation}

Notice that the value of $\xi_{2\infty}$ leads us to define $A_{2\infty}$ and $B_{2\infty}$ in an explicit form as well: 
\begin{eqnarray}
&&A_{2\infty}=- \frac{2\, \left(\mathrm{Ste} + 3\right)}{\mathrm{Ste}} +\frac{12\, \left(2\, \mathrm{Ste} + 3\right)}{\mathrm{Ste}\, \left(  \sqrt{\Ste^2+24\Ste+36}-\mathrm{Ste}\right)},\label{A2Inf}\\
&&B_{2\infty}= \frac{3(\, \mathrm{Ste} + 2)}{\mathrm{Ste}} -\frac{12(2\Ste+3)}{\mathrm{Ste}\, \left( \sqrt{\Ste^2+24\Ste+36}-\mathrm{Ste}\right)}.\label{B2Inf}
\end{eqnarray}

The above reasoning can be summarized in the following theorem:

\begin{teo} The solution to the problem \emph{(P2$_\infty$)}, for a quadratic profile in space, is unique and it is given by (\ref{TempP2Inf}) and (\ref{FrontP2Inf}) where the positive constants $A_{2\infty}$, $B_{2\infty}$ and $\xi_{2\infty}$ are defined explicitly from the data of the problem by the expressions (\ref{A2Inf}), (\ref{B2Inf}) and (\ref{Xi-2Inf-valor2}) respectively.

\end{teo}

The fact of having the exact solution to problem (P2$_{\infty}$) allows us to proof the following result:
\begin{teo}\label{Teo:ConvP2} The solution to problem \emph{(P2)} converges to the solution to  problem \emph{(P$2_{\infty}$)} when \emph{Bi}$ \rightarrow \infty$. 
\end{teo}
\begin{proof}
It is analogous to the proof given in Theorem \ref{Teo:ConvP1}.
\end{proof}

\textbf{Problem (P${3_{\infty}}$)}. Find the temperature $T_{3\infty}=T_{3\infty}(x,t)$ at the solid region $0<x<s_{3\infty}(t)$  and the location of the free boundary  $x=s_{3\infty}(t)$ such that:
\begin{align} 
&\int\limits_0^{s_{3\infty}(t)} \int\limits_0^x \frac{\partial T_{3\infty}}{\partial t}(\xi,t) d\xi dx = -\dfrac{k}{\rho c} \left[T_{3\infty}(0,t)+\right.  \qquad\qquad\qquad\qquad\nonumber\\
&\qquad\qquad\qquad\qquad\quad+\left.\frac{\partial T_{3\infty}}{\partial x}(0,t)s_{3\infty}(t) \right]~, \ 0<x<s_{3\infty}(t)~, \  t>0~,  \\
&T_{3\infty}(0,t)=-\Theta_{\infty}~, \ t>0~, \\
&T_{3\infty}(s_{3\infty}(t),t)= 0~, \ t>0~, \\
&k\dfrac{\partial T_{3\infty}}{\partial x}(s_{3\infty}(t),t)=\rho \lambda \dot{s}_{3\infty}(t)~, \ t>0~, \\
& s_{3\infty}(0)=0~. 
\end{align}

The solution to problem (P3$_{\infty}$) is given by:
\begin{eqnarray}
T_{3\infty}(x,t)&=& -A_{3\infty}\Theta_{\infty}\left(1-\dfrac{x}{s_{3\infty}(t)}\right) -B_{3\infty}\Theta_{\infty} \left(1-\dfrac{x}{s_{3\infty}(t)}\right)^{2}~, \label{TempP3Inf}\\
s_{3\infty}(t)&=&2 \xi_{3\infty} \sqrt{\alpha t}~,\label{FrontP3Inf}
\end{eqnarray}
where the constants $A_{3\infty}$ and $B_{3\infty}$ are:
\begin{eqnarray}
A_{3\infty}&=&1-\frac{\xi_{3\infty}^2}{3}~,\\
B_{3\infty}&=&\frac{1}{3}\xi_{3\infty}^{2}~,
\end{eqnarray}

As we impose that $A_{3\infty}$ and $B_{3\infty}$ must be positive, it is easily obtained that $\xi_{3\infty}$ must belong to the interval $(0,\sqrt{3})$. Therefore $\xi_{3\infty}$ is a positive solution of the second degree polynomial equation:
\begin{eqnarray} 
\left(6+\text{Ste}\right)z^{2}-3\text{Ste}=0~, \qquad 0<z<\sqrt{3}, \label{Xi-3Inf}
\end{eqnarray}

It it clear that the polynomial function $p_{3\infty}=p_{3\infty}(z)$ defined by the l.h.s of (\ref{Xi-3Inf}) has a unique root in $\mathbb{R}^{+}$.
As we need that this root be in a required range, we study the sign of $p_{3\infty}$ in the extremes of the interval:
\begin{eqnarray*}
p_{3\infty}(0)&=& -3\text{Ste}<0,\\
p_{3\infty}\left(\sqrt{3} \right)&=& 18>0.
\end{eqnarray*}

Then,  we can assure that the unique positive root of $p_{3\infty}$ belongs to the interval $(0,\sqrt{3})$ and it is given by:
\begin{equation}
\xi_{3\infty}=\sqrt{\frac{3\text{Ste}}{6+\text{Ste}}}. \label{Xi-3Inf-valor}
\end{equation}

Hence, $A_{3\infty}$ and $B_{3\infty}$ become explicitly defined by:
\begin{equation}
A_{3\infty}=\dfrac{6}{6+\Ste}, \qquad \qquad B_{3\infty}=\dfrac{\Ste}{6+\Ste}.\label{A3B3Inf}
\end{equation}

The following theorem summarizes the previous reasoning:

\begin{teo} The solution to the problem \emph{(P3$_\infty$)}, for a quadratic profile in space, is unique and it is given by (\ref{TempP3Inf}) and (\ref{FrontP3Inf}) where the positive constants $A_{3\infty}$, $B_{3\infty}$ and $\xi_{3\infty}$ are defined explicitly from the data of the problem by the expressions (\ref{A3B3Inf}) and (\ref{Xi-3Inf-valor}) respectively.

\end{teo}

\begin{teo} \label{teor 4.6} The solution to  problem \emph{(P3)} converges to the solution to problem \emph{(P$3_{\infty}$)} when $\emph{Bi}\rightarrow \infty$. 
\end{teo}
\begin{proof}
It is analogous to the proof of Theorem \ref{Teo:ConvP1}.
\end{proof}

\textbf{Problem (P${4_{\infty}}$)}. Find the temperature $T_{4\infty}=T_{4\infty}(x,t)$ at the solid region $0<x<s_{4\infty}(t)$  and the location of the free boundary  $x=s_{4\infty}(t)$ such that:
\begin{align} 
&\int\limits_0^{s_{4\infty}(t)} \int\limits_0^x \frac{\partial T_{4\infty}}{\partial t}(\xi,t) d\xi dx = -\dfrac{k}{\rho c} \left[T_{4\infty}(0,t)+\right.  \quad\qquad\qquad \qquad \\
&\left.\qquad \qquad\qquad\qquad\quad+\frac{\partial T_{4\infty}}{\partial x}(0,t)s_{4\infty}(t) \right]~, \ 0<x<s_{4\infty}(t)~, \ t>0~,  \\
&T_{4\infty}(0,t)=-\Theta_{\infty}~, \ t>0~,\\
&T_{4\infty}(s_{4\infty}(t),t)= 0~, \ t>0~,\\
&\left(\frac{\partial T_{4\infty}}{\partial x}\right)^2(s_{4\infty}(t),t)=-\frac{\lambda}{c}\frac{\partial ^2T_{4\infty}}{\partial x^2}(s_{4\infty}(t),t)~, \ t>0~, \\
&s_{4\infty}(0)=0. 
\end{align} 

The solution to problem (P4$_{\infty}$) is given by:
\begin{eqnarray}
T_{4\infty}(x,t)&=&-A_{4\infty}\Theta_{\infty}\left(1-\dfrac{x}{s_{4\infty}(t)}\right) -B_{4\infty}\Theta_{\infty} \left(1-\dfrac{x}{s_{4\infty}(t)}\right)^{2} \label{TempP4Inf}\\
s_{4\infty}(t)&=&2 \xi_{4\infty} \sqrt{\alpha t}~,\label{FrontP4Inf}
\end{eqnarray}
where the constants $A_{4\infty}$ and $B_{4\infty}$ are given by:
\begin{eqnarray}
A_{4\infty}&=&1-\frac{\xi_{4\infty}^{2}}{3}~,\\
B_{4\infty}&=&\frac{1}{3}\xi_{4\infty}^{2}>0.
\end{eqnarray}

As in problem (P3$_{\infty}$), the coefficients $A_{4\infty}$ and $B_{4\infty}$ must be positive. Therefore it turns out that $\xi_{4\infty}$ must belong to the interval $(0,\sqrt{3})$.
From this fact we have that $\xi_{4\infty}$ is a positive solution of the fourth degree polynomial equation:
\begin{eqnarray} 
\text{Ste} \ z^{4}-6\left(1+\text{Ste}\right)z^{2}+9\text{Ste}=0~,\quad 0<z<\sqrt{3}~, \label{Xi-4Inf}
\end{eqnarray}

Let us call by $p_{4\infty}=p_{4\infty}(z)$ the l.h.s of equation (\ref{Xi-4Inf}). Notice first that $p_{4\infty}$ can have at most two roots in $\mathbb{R}^+$. In addition:
\begin{eqnarray*}
p_{4\infty}(0)=9\Ste>0~, \qquad  p_{4\infty}(\sqrt{3})=-18<0~, \qquad p_{4\infty}(+\infty)=+\infty.
\end{eqnarray*}

Consequently  it becomes that $p_{4\infty}$ has a unique solution on the interval $(0,\sqrt{3})$ which can be explicitly defined by:
\begin{equation}
\xi_{4\infty}=\left(\dfrac{3\left(\left(1+\text{Ste}\right)-\sqrt{2\text{Ste}+1}\right)}{\text{Ste}}\right)^{1/2}~. \label{Xi-4Inf-valor2}
\end{equation}

Thus we get:
\begin{equation}
A_{4\infty}=\dfrac{\sqrt{2\Ste+1}-1}{\Ste}~,\qquad  B_{4\infty}=\dfrac{1+\Ste-\sqrt{2\Ste+1}}{\text{Ste}},\label{A4B4Inf}
\end{equation}
and therefore we have the following theorem:

\begin{teo} The solution to the problem \emph{(P4$_\infty$)}, for a quadratic profile in space, is unique and it is given by (\ref{TempP4Inf}) and (\ref{FrontP4Inf}) where the positive constants $A_{4\infty}$, $B_{4\infty}$ and $\xi_{4\infty}$ are defined explicitly from the data of the problem by the expressions (\ref{A4B4Inf}) and (\ref{Xi-4Inf-valor2}) respectively.
\end{teo}

Consequently, we have the following convergence theorem:

\begin{teo} The solution to problem \emph{(P4)} converges to the solution to problem \emph{(P$4_{\infty}$)} when \emph{Bi}$\rightarrow \infty$.
\end{teo}
\begin{proof}
It is analogous to the proof given in Theorem \ref{Teo:ConvP1}. 
\end{proof}

\section{Comparisons with the exact solution}

In this section comparisons between the known exact solution of the Stefan problem (P) and the approximate solutions associated to each problem (Pi), $i=1,2,3,4$ are done. Furthermore, in view of the convergence results presented in Section 4, we compare the exact solution to problem (P) where a convective boundary condition at the fixed face is imposed with the exact solution of problem (P$_\infty$) in which a Dirichlet boundary condition is taken into account. The same is done with each approximate solution obtained from problem (Pi) and (Pi$_\infty$), $i=1,2,3,4$.

\subsection{Comparisons between the free boundaries}

To compare the free boundaries obtained in each problem,  we compute the coefficient that characterizes the free boundary in the exact problem (P) and the approximate problems (Pi) $i=1,2,3,4$. The exact value of $\xi$ and the different approaches $\xi_i$ $i=1,2,3,4$  are obtained by solving the equations obtained in Section 2 and 3.

In the cases of  problems (P$_\infty$) and (Pi$_\infty$) we compute: the exact value $\xi_{\infty}$,  by solving a nonlinear equation, and the different approaches $\xi_{i\infty}$, $i=1,2,3,4$  by closed expressions that were obtained in Section 4.

In order to make visible the accuracy of each approximate method; fixing some Stefan and Biot numbers we calculate for $i=1,2,3,4$: the relative error of the free boundary associated to problem (Pi) given by $\text{E}_{\text{rel}}(s_i)=\vert \tfrac{\xi_i-\xi}{\xi}\vert$ and the relative error associated to problem (Pi$_\infty$) defined by \mbox{$ \text{E}_{\text{rel}}(s_{i\infty})=\vert \tfrac{\xi_{i\infty}-\xi_{\infty}}{\xi_{\infty}}\vert $}.

Varying Ste and Bi numbers, we obtain different visualizations for the behaviour of the relative error committed in each method. Since similar plots were obtained for a great number of cases, we present here the most significant ones. For Ste$=10^{-3}, 1,10$ we compute $\text{E}_{\text{rel}}(s_{i\infty})$ for $i=1,2,3,4$. Also, for those  Ste numbers and making  Bi vary between 0  and  a suitable value in which the convergence can be appreciated, we compute $\text{E}_{\text{rel}}(s_i)$, $i=1,2,3,4$. Figures 1, 2 and 3 shows the behaviour of the relative errors that correspond to the problem (Pi) and (Pi$_\infty$) for $i=1,2,3$. The plot that corresponds to (P4) and (P4$_\infty$) is presented separately so that the error is better noticed (see Figure 4).

 \begin{figure}[h!]
 \centering
 \includegraphics[scale=0.3]{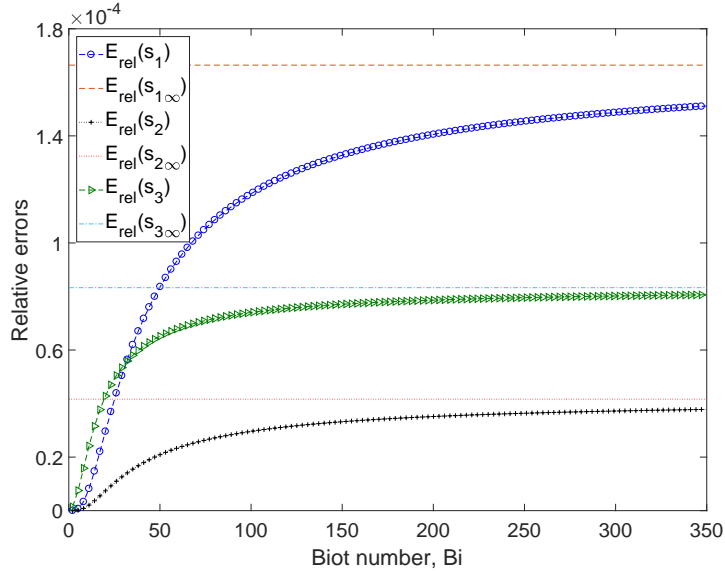}
\caption{Relative errors of $s_1$, $s_2$ and $s_3$ against $\Bi$ for Ste$=10^{-3}$.}

 \end{figure}

\begin{figure}[h!]

 \centering
 \includegraphics[scale=0.3]{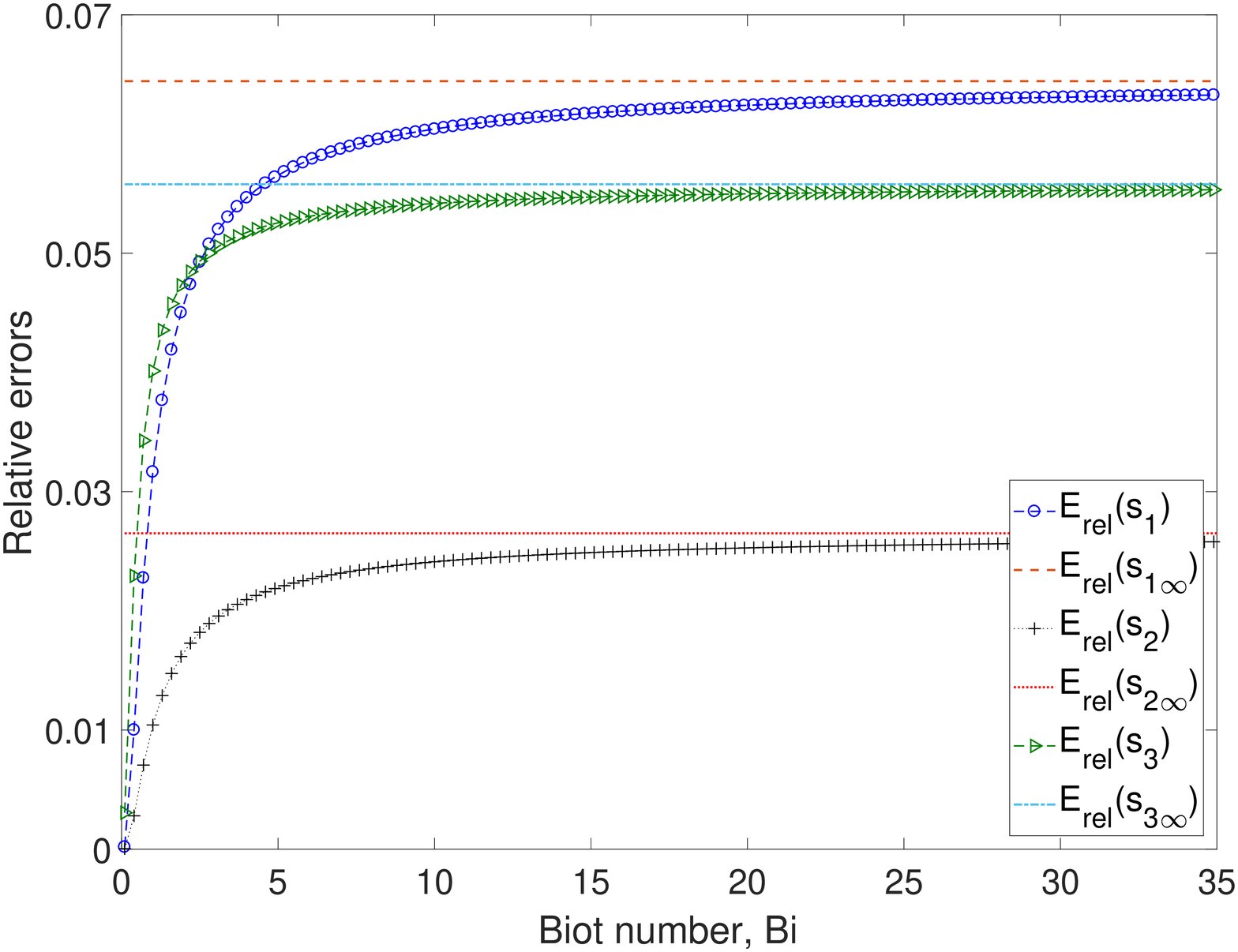}
 \caption{ Relative errors of $s_1$, $s_2$ and $s_3$ against $\Bi$ for Ste$=1$.}

\end{figure}

\begin{figure}[h!]

 \centering 
 \includegraphics[scale=0.3]{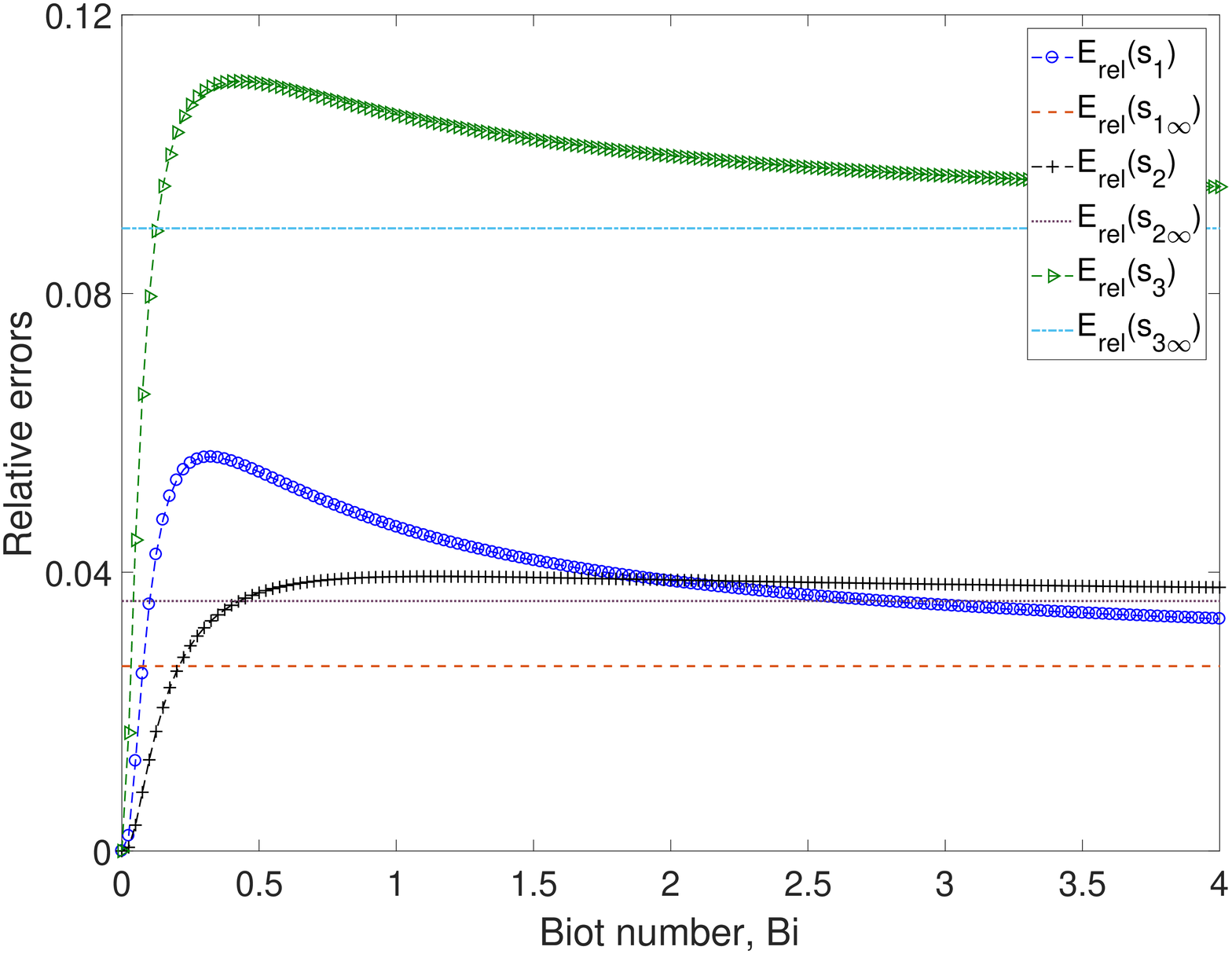}
 \caption{ Relative errors of $s_1$, $s_2$ and $s_3$ against $\Bi$ for Ste$=10$.}

\end{figure}

\begin{figure}[h!]

 \centering
 \includegraphics[scale=0.3]{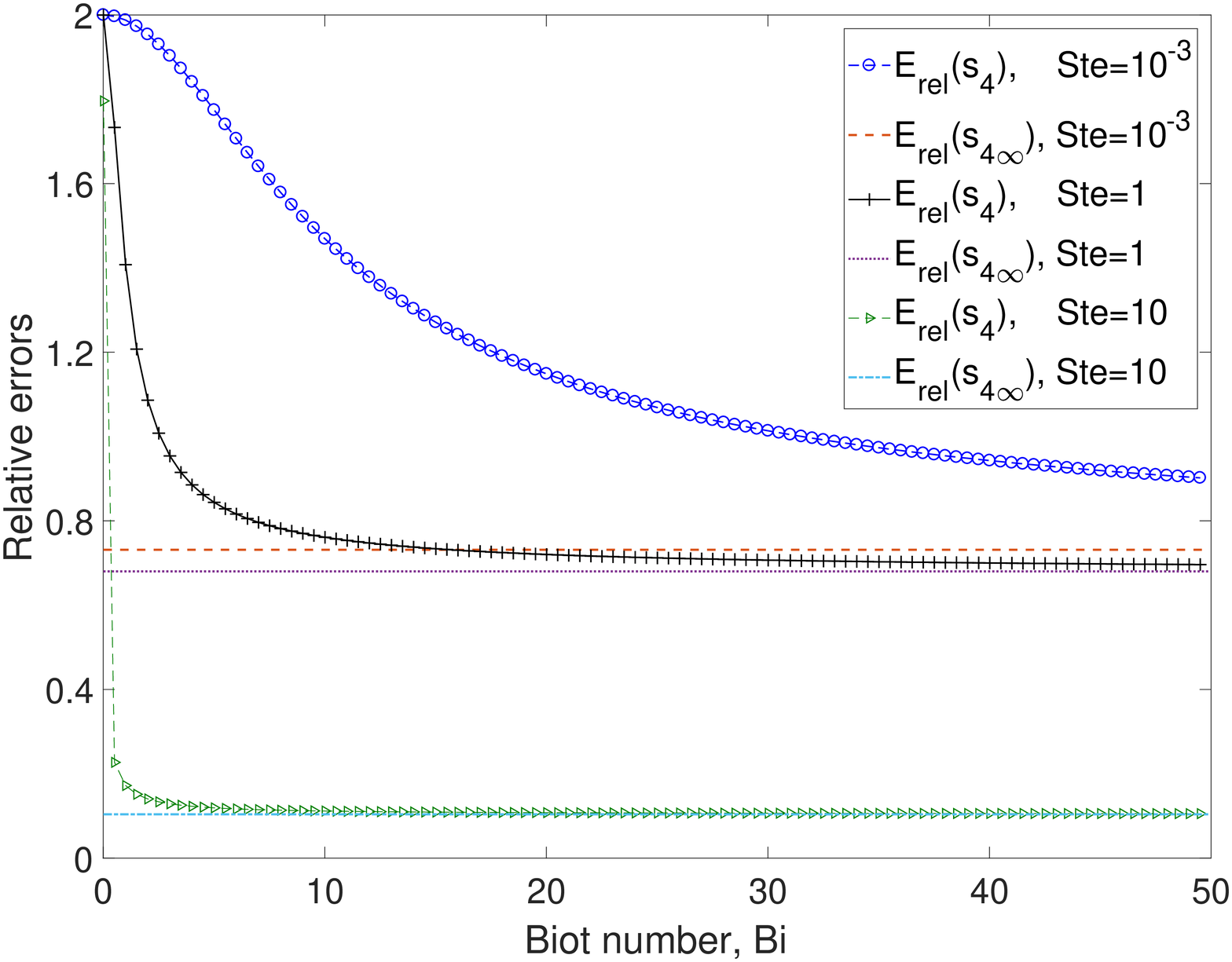}
 \caption{ Relative errors of $s_4$ against $\Bi$ for  Ste$=10^{-3}, 1, 10$.}

\end{figure}

From the numerical analysis we observe that for Ste $<1$ and varying Bi number, (P2) and (P3) gives the best approximation of the free boundary.

In addition, (P4) does not constitute a good approximation.
It is worth to mention that taking a Stefan number up to about 1 covers most of phase change materials.

In the cases for Ste $>1$ we see that the best approximations are given by (P1) and (P2).

%In all cases we observe that the best approximate method is the one given by the resolution of (P2) in which the committed error is  about \%2.

\subsection{Comparisons between the temperature profiles}

Considering the 1D solidification process in the region $x>0$, we fix the physical parameters for ice given by: the thermal conductivity $k=2.219$ W/(m$^{\circ}$C), the specific heat $c=2097.6$ J/(kg K), the thermal diffusivity $\alpha=1.15\times 10^{-6}$ m$^2$/s and  the latent heat of fusion $\lambda=3.33\times 10^5$ J/kg.

In order to show the behaviour of the  temperature  of problems (P), (Pi), i$=1,2,3,4$ with respect to the space variable $x$, we fix the time at $t=10$ s, the neighbourhood
 temperature at the fixed face $\Theta_{\infty}=5~^{\circ}$C and the coefficient that characterizes the heat transfer at the fixed face \mbox{$h=1.65\times 10^5$ W$\sqrt{\text{s}}$/(m$^{\circ}$C)}. Notice that this choices lead us to define the Stefan number and Biot number as: $\Ste=0.0314$ and Bi$=80$.

In Figure 5 we compare the four approximate temperatures given by $T_i$ ($i=1,2,3,4$) with the exact solution $T$ given by (\ref{TempExacta}). In addition we plot the temperature $T_{\infty}$ with the purpose of showing the convergence when Bi is large.
\begin{figure}[h]
 \centering
 \includegraphics[scale=0.3]{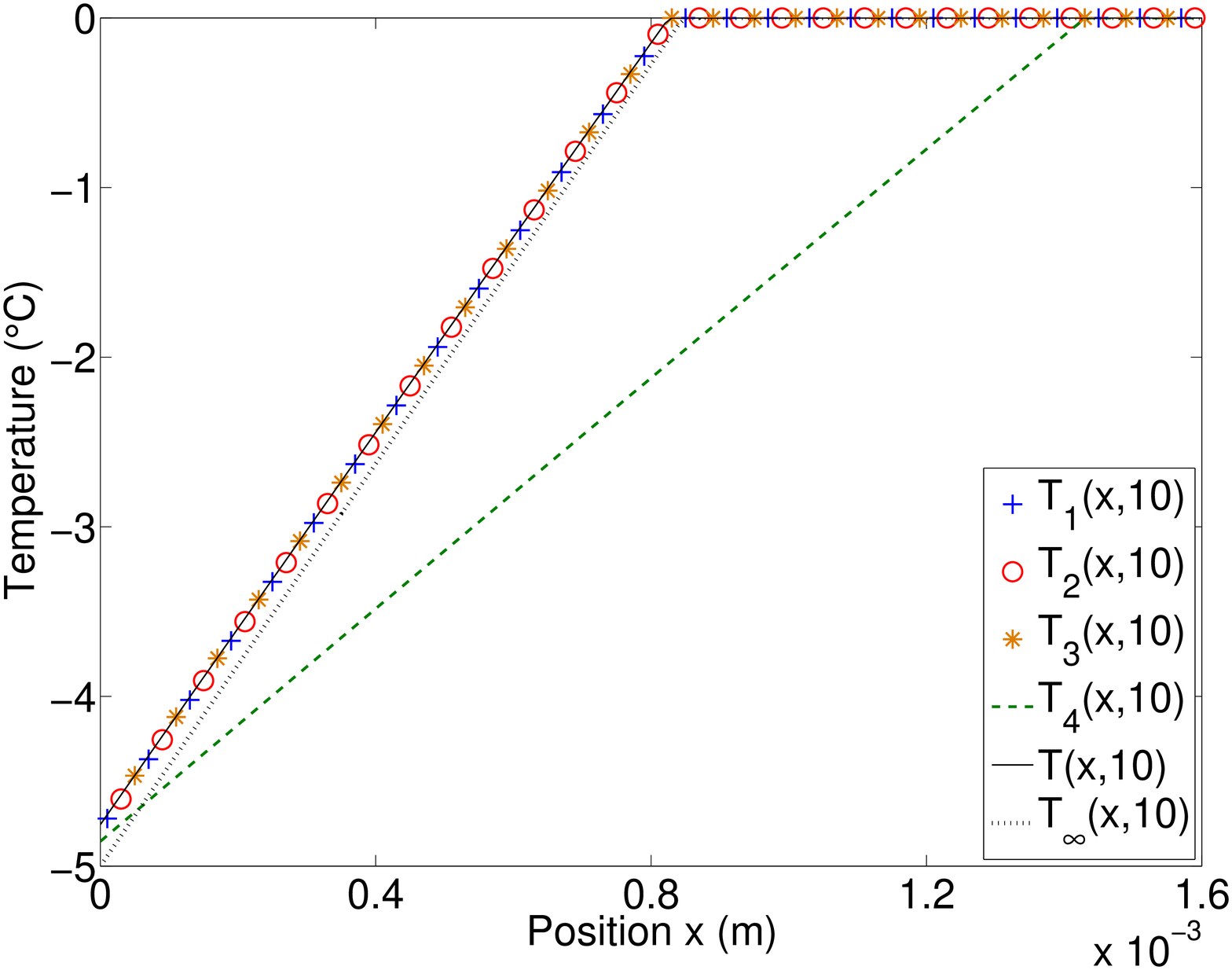}
 \caption{ Plot of the profile temperatures at $t$=10 s for Ste=0.0314 and Bi=80.}

\end{figure}

Before giving results, notice that in Figure 5 we can  appreciate that neither the plot of the exact temperature corresponds to an error function  nor the plot of the approximate ones corresponds to quadratic functions in space. The main reason of this fact is because the range of $x$ (position) is not large enough to perceive this type of curves.

From Fig. 5, we can  easily observe that (P4) gives the worst approximation for the temperature; whereas the approaches obtained by problems (P1), (P2) and (P3) are difficult to distinguish from each other and more importantly from the exact solution. 
As a consequence to make the difference between the curves clearer we present the following tables of absolute errors defined by \mbox{$E_{\text{abs}}(T_i)=\lvert T(x,10)-T_i(x,10)\rvert$}, for $i=1,2,3,4$.
\vspace{0.1cm}

\begin{center}
\begin{footnotesize}
Table 1: Absolute errors of temperature approximations at $t=10$s.
\end{footnotesize}
\begin{tabular}{c|c|c|c|c}
\hline
        position $x$    & $E_{\text{abs}}(T_1)$    & $E_{\text{abs}}(T_2)$& $E_{\text{abs}}(T_3)$& $E_{\text{abs}}(T_4)$\\
        \hline
       0    	 &	0.023661  &   0.024332  &   0.000581      & 0.0993 \\
       0.0001    &  0.018960   &  0.021135  &   0.002256      & 0.3339  \\
       0.0002    & 	0.016015  &   0.018719  &   0.004368      & 0.5690   \\
       0.0003    & 	0.014576  &   0.016830  &   0.006667      & 0.8042\\
       0.0004    & 	0.014391  &   0.015219  &   0.008902      & 1.0395\\
       0.0005    & 	0.015212  &   0.013637  &   0.010823      & 1.2744\\
       0.0006    & 	0.016789  &   0.011835  &   0.012183      & 1.5088\\
       0.0007    & 	0.018877  &   0.009567  &   0.012735	  & 1.7424\\
       0.0008    & 	0.021231  &   0.006587  &   0.012234      & 1.9749\\
       0.0009    &        0   &         0   &         0       & 1.7843\\
        0.001    &        0   &         0   &         0       & 1.4467\\
 \hline
\end{tabular}

\vspace{0.1cm}
\begin{center}
\begin{footnotesize}
Table 2: Absolute errors of temperature approximations at $t=10$s.
\end{footnotesize}
\begin{tabular}{c|c|c|c|c}
\hline
       position $x $     & $E_{\text{abs}}(T_1)$    & $E_{\text{abs}}(T_2)$& $E_{\text{abs}}(T_3)$& $E_{\text{abs}}(T_4)$\\
        \hline
  %   0.000815  &   0.021592  &  0.0060647  &   0.012053   &    2.0097 \\
  %   0.000816  &   0.021616  &  0.0060290  &   0.012040   &    2.0120\\
  %   0.000817  &   0.021640  &  0.0059932  &   0.012026   &    2.0144\\
  %   0.000818  &   0.021664  &  0.0059574  &   0.012013   &    2.0167\\
  %   0.000819  &   0.021688  &  0.0059214  &   0.011999   &    2.0190\\
     0.000820  &   0.021712  &  0.0058853  &   0.011986   &    2.0213\\
     0.000821  &   0.021736  &  0.0058492  &   0.011972   &    2.0236\\
     0.000822  &   0.021760  &  0.0058129  &   0.011958   &    2.0259\\
     0.000823  &   0.021784  &  0.0057766  &   0.011944   &    2.0283\\
     0.000824  &   0.021809  &  0.0057401  &   0.011930   &    2.0306\\
     0.000825  &   0.021833  &  0.0057035  &   0.011916   &    2.0329\\
     0.000826  &   0.021187  &  0.0049971  &   0.011231   &    2.0345\\
     0.000827  &   0.015511  &          0  &   0.005516   &    2.0312\\
     0.000828  &   0.009834  &          0  &          0   &    2.0278\\
     0.000829  &   0.004158  &          0  &          0   &    2.0244\\
     0.000830  &          0  &          0  &          0   &    2.0210\\
     \hline
\end{tabular}
\end{center}

\noindent 
\end{center}

In Table 1 we can see that the accuracy of the temperatures corresponding to problems (P1), (P2) and (P3) is lower than 0.025.

If we analyse the absolute error between $x=0.00082$ and $x=0.00083$ (Table 2), we obtain that the problem (P2) gives the most suitable approach to the free boundary, as we have already seen in the previous subsection.

\section{Conclusion}

In this paper, approximate analytical solutions by using some variations of the classical heat balance integral method are obtained  for a one-dimensional one-phase Stefan problem (P)  with a convective (Robin) \mbox{boundary} condition at the fixed face. 

This work provides information about which approximate method to choose for approaching the solution to the solidification problem (P), according to the data of the problem (Ste and Bi number). 

The approximate techniques  for tracking the free boundary used throughout this work corresponds to the heat balance integral method (P1); an alternative of the heat balance integral method (P2); the refined integral method (P3) and an alternative of the refined integral method (P4). The alternatives forms proposed here are inspired in \cite{Wo2001} and consist in re-working or not the Stefan condition (\ref{CondStefan}).

In most of cases, due to the nonlinearity of the Stefan problems there is no analytical solution. In case of problem (P) the exact solution is available recently in the literature \cite{Ta2017}, so that the main advantage is that we can compare the different approaches with the exact solution and consequently estimate the accuracy of the different approximate methods.

Therefore it can be said that in general the optimal approximate technique for solving (P) is given by the alternative form of the heat balance integral method defined by (P2), in which the Stefan condition is not removed and remains equal to the exact problem.

\section*{Acknowledgements}

The present work has been partially sponsored by the Project PIP No 0275 from CONICET-UA,
Rosario, Argentina, and ANPCyT PICTO Austral No 0090.

%%=========================================

%\section*{References}
%\bibliographystyle{elsarticle-num} 
%\bibliography{References.bib}

\end{document}